\documentclass[12pt]{article}

\usepackage{latexsym,amssymb,amsmath}
\usepackage[dvips]{graphicx}
\usepackage{bm}
\newcommand{\Z}{\mathbb Z}
\newcommand{\N}{\mathbb N}
\newcommand{\R}{\mathbb R}
\newcommand{\Q}{\mathbb Q}

\newcommand{\sma}{\left(\begin{array}}
\newcommand{\fma}{\end{array}\right)}

\newtheorem{lem}{Lemma}[section]
\newtheorem{defn}[lem]{Definition}

\newtheorem{co}[lem]{Corollary}
\newtheorem{thm}[lem]{Theorem}
\newtheorem{prop}[lem]{Proposition}

\newenvironment{proof}{\textbf{Proof.}}{\newline\hspace*{\fill}{$\Box$}\\}

\begin{document}
\title{Generalised Baumslag-Solitar groups and Hierarchically
  Hyperbolic Groups}

\author{J.\,O.\,Button}
\date{}
\newcommand{\Address}{{
  \bigskip
  \footnotesize

\textsc{Selwyn College, University of Cambridge,
Cambridge CB3 9DQ, UK}\par\nopagebreak
  \textit{E-mail address}: \texttt{j.o.button@dpmms.cam.ac.uk}
}}

\maketitle
\begin{abstract}
  We look at isometric actions on arbitrary hyperbolic spaces of
  generalised Baumslag - Solitar groups of any rank
  (the rank of the free abelian vertex and edge subgroups).
It is known that being a hierarchically
hyperbolic group is not a quasi-isometric invariant. We show that
virtually being a hierarchically hyperbolic group is not invariant
under quasi-isometry either, and nor is property (QT).
\end{abstract}

\section{Introduction}

Gromov's notion of a word hyperbolic group encapsulates what it means
for a finitely presented group to be negatively curved. In particular this
property is not just a commensurability invariant (that is, if $H$ is a finite
index subgroup of $G$, for which we write $H\leq_f G$, then $H$ has the
property if and only if $G$ has the property) but it is also invariant
under quasi-isometry.

However trying to come up with an equivalent notion of
non positive curvature for finitely presented groups which is also invariant
under quasi-isometry seems less clear. One
definition that is invariant is that of having (at most) quadratic
Dehn function but this contains groups that would not naturally be thought of
as being non positively curved. To give but one example, in \cite{dects}
it is shown that for each $n\geq 2$
there exists a metabelian group with quadratic Dehn function
but which contains the Baumslag - Solitar group $BS(1,n)$, whereas one might
hope that any finitely generated soluble subgroup of a non positively
curved group is virtually abelian.

Even if invariance under quasi-isometry fails, one could at least hope for a
definition of non positive curvature which is invariant under
commensurability. However for the property of being CAT(0) (a group
which acts geometrically on a CAT(0) space), which is known not to be
invariant under quasi-isometry, it is currently open whether $H$ being a
CAT(0) group and $H\leq_f G$ implies that $G$ is a CAT(0) group. It is also
still open whether all hyperbolic groups are CAT(0).

A more recent property that also aims to encapsulate non positive curvature
is that of being a hierarchically hyperbolic group (a HHG for short).
This does include all hyperbolic groups but it was recently shown in
\cite{petspr} that this is not a commensurability invariant as there are
groups which are not HHGs but which have a finite index subgroup equal
to $\Z^2$ which is. Note that neither CAT(0) groups or HHGs contain each
other.

However there is an easy way to turn a group theoretic property $\cal P$
into a commensurability invariant provided it is preserved by finite index
subgroups (which many properties are, including being CAT(0) and an HHG).
We merely alter it to being virtually $\cal P$, that is it has some finite
index subgroup with $\cal P$. We can thus ask instead: for finitely
generated groups, is being virtually $\cal P$ a quasi-isometric invariant?
For instance this question has been considered for the properties of being
virtually free, virtually cyclic, virtually abelian, virtually nilpotent,
virtually polycyclic and virtually soluble. None of these would be
commensurability invariants if virtually were removed, but the first four
are known to be quasi-isometry invariants by various deep results. The fifth
is unknown and the sixth is false for finitely generated groups but
unknown for finitely presented groups (see \cite{delah} IV.B.50 and the
references given there).

Therefore it is natural to ask in this context whether virtually being
an HHG is a quasi-isometry invariant. Note that virtually being a CAT(0)
group is not. This is demonstrated by the well known example of a group $G$
which is a central extension of $\Z$ by a closed surface group $S_g$
for $g\geq 2$ which
does not virtually split and so is not a CAT(0) group (see \cite{bh} II.7.26)
but which is quasi-isometric to $\Z\times S_g$. This group $G$ has all of
its finite index subgroups of the same form, so it is not virtually a
CAT(0) group either. We also note here that $\Z\times S_g$ acts geometrically
not just on a CAT(0) space but on a CAT(0) cube complex too, thus the
property $\cal P$ of acting geometrically on a CAT(0) cube complex is not a
quasi-isometric invariant and nor is being virtually $\cal P$ (note that
here virtually $\cal P$ is not equal to $\cal P$ by \cite{hag}).

However the group $G$, and more generally central extensions of $\Z$ by 
any non elementary hyperbolic group, was recently shown to
be a HHG in \cite{hrsspr} Corollary 4.3, so will not provide a counterexample
to virtually HHGs being invariant under quasi-isometry. Nor will the examples
which are not HHGs in \cite{petspr} Theorem 4.4 work because these are all
virtually $\Z^k$ for some $k$.

In this paper we show in Corollary \ref{hhglm} that there exist finitely
presented groups with no finite index subgroup being a HHG but
which are quasi-isometric to an HHG. Thus even adding virtually to the
property of being an HHG will not make it a quasi-isometric invariant.
The particular group used in Corollary \ref{hhglm} is the Leary - Minasyan group
first appearing in \cite{lrmn}. A Leary - Minasyan group
denotes any group formed by starting with $A=\Z^n$ for $n\geq 1$ which is
regarded as sitting naturally in $\Q^n$ and then taking
a matrix $M:\Q^n\rightarrow\Q^n$ which is in $GL(n,\Q)$ and a finite index
subgroup $B$ of $A \cap M^{-1}(A)$. The group is then defined to be the HNN
extension with base $A$ and associated subgroups
$B$ and $M(B)$ using the matrix $M$. We refer to THE Leary - Minasyan group
$L$ as the case where $n=2$, 
\[M=\sma{rr} \frac{3}{5}& -\frac{4}{5}\\  \frac{4}{5}&\frac{3}{5}\fma\]
and $B=A\cap M^{-1}(A)$ which is a CAT(0) group. It was shown in \cite{lrmn}
Theorem 1.1 that a Leary - Minasyan group is biautomatic if and only if it is
virtually biautomatic which occurs if and only if the matrix $M$ has finite
order. (It is unknown whether there exist virtually biautomatic groups
which are not biautomatic.) In particular the group $L$ is CAT(0)
but not (virtually) biautomatic: this was the first known example.
As for whether every HHG is biautomatic, this was answered negatively in
\cite{hghv}. In the introduction, it is asked whether any non-biautomatic
Leary - Minasyan group is an HHG, with the answer expected to be no.
In this paper we confirm this in Corollary \ref{nthhg}
by showing that a Leary - Minasyan group is virtually an HHG if and
only if the matrix $M$ has finite order.

It turns out that our arguments work in a somewhat wider class of groups
than Leary - Minasyan groups,
that of generalised Baumslag - Solitar
groups of arbitrary rank. A generalised Baumslag - Solitar group of
rank $n$, or $GBS_n$ group for short,
is a finite graph of groups where all vertex and edge groups
are isomorphic to $\Z^n$ (sometimes the term generalised Baumslag - Solitar
group refers only to the case $n=1$ but we will use the term in its wider
context). A $GBS_n$ group $G$ gives rise to a homomorphism from $G$ to
$GL(n,\Q)$, defined up to conjugacy,
which we call the modular homomorphism of $G$. (Strictly speaking
this should be the modular homomorphism of the decomposition of $G$ as a
$GBS_n$ group but in all but some basic cases this homomorphism is unique,
even if the decomposition is not.) In the case of
a Leary - Minasyan group the graph of groups has one vertex and one edge
with the modular homomorphism sending the stable letter to the matrix $M$
and all of the vertex subgroup $\Z^n$ to the identity.

We begin in Section 2 with some introductory material on how groups act
on hyperbolic spaces. In Section 3 we introduce generalised Baumslag - Solitar
groups $G$ of arbitrary rank $n$ and define the aforementioned modular
homomorphism. We use this to examine the free abelianisation of $G$, that
is the usual abelianisation $G/G'$ but with the torsion of $G/G'$
quotiented out. This gives us a dichotomy between groups with finite and
infinite monodromy, where the monodromy is the image of $G$ in $GL(n,\Q)$
under the modular homomorphism. Indeed it is shown in Theorem \ref{fnmn}
that a $GBS_n$ group with finite monodromy is virtually equal to
$\Z^n\times F_r$ for some finite rank free group $F_r$ (and hence in this
case $G$ is virtually a HHG), whereas a $GBS_n$ group with infinite
monodromy cannot be of this form. (As a consequence, whether a $GBS_n$
group has finite or infinite monodromy depends in all cases only on the group
and not the decomposition.)

In Subsection 3.3 we use this information to examine the possible actions
by isometries of a $GBS_n$ group $G$ on an arbitrary hyperbolic space. Whilst
a group of the form $\Z^n\times F_r$ will have many different actions on
hyperbolic spaces, we show in Theorem \ref{act}
that if $G$ has infinite monodromy then there is a non-trivial element of
the base $\Z^n$-subgroup which never acts loxodromically in any action
of $G$ on a hyperbolic space.

As for particular actions on hyperbolic spaces, the concept of an
acylindrical action of a group $G$ on an arbitrary hyperbolic space
$X$ was introduced in \cite{bow}. However if $X$ is bounded then any
action by isometries is acylindrical, so we require acylindrical actions
of $G$ with unbounded orbits in order to obtain anything useful. This
was developed in \cite{osac} to obtain the much studied concept of an
acylindrically hyperbolic group $G$, in which $G$ has an acylindrical
action on some hyperbolic space that is non - elementary, which
is equivalent to saying that the action has unbounded orbits and
$G$ is not virtually cyclic. In the
theory of HHGs, such a group comes equipped with various hyperbolic
spaces, including one on which $G$ acts by isometries called the
maximal domain $S$. It was shown in \cite{bhs1} Theorem 14.3 that if $G$ is an
HHG then the resulting action of $G$ on $S$ is acylindrical. However
it is perfectly possible that $S$ is a bounded metric space even if $G$
is an infinite HHG (for instance for groups acting geometrically on the
product of two trees).

Now it can be seen using standard results that $GBS_n$ groups are never
acylindrically hyperbolic but this does not allow us to conclude that
they are not HHGs as well, because it could be that in some potential
HHG structure on such a group, the maximal domain $S$ is bounded. However this
can be avoided by appealing to recent results appearing in \cite{petspr}
which show that if $G$ is any HHG then
there is a finite collection of unbounded domains, invariant under the
action of $G$ and which are pairwise orthogonal, such that any unbounded
domain in the HHG structure is nested in one of these. If $G$ is an infinite
group then this collection will be non empty. Moreover the action of $G$
on each of these domains can be combined to give an action of $G$ on their
product and this action is also acylindrical (as well as containing
a loxodromic element). Of course a product
of at least two unbounded hyperbolic spaces will not itself be hyperbolic
but the definition of an acylindrical action makes sense for an arbitrary
metric space. Whilst we do not know of any group theoretic consequences
if there exists an acylindrical action with a loxodromic element
on a product $P$ of hyperbolic spaces, in Section 4 we look at such actions
which also preserve the factors of $P$, which we call a product
acylindrical action.
We show in Theorem \ref{nop} that every element in such an action
must either be elliptic or loxodromic, as was shown to be the case
for a single hyperbolic space in \cite{bow} Lemma 2.2. 

This is then used to obtain Theorem \ref{pacg} which states that a $GBS_n$
group with infinite monodromy has no product acylindrical action. In
Section 5 we give some background on the properties of HHGs that we will
need, as opposed to giving the full definition. We then conclude in
Corollary \ref{nthhg} that a $GBS_n$ group is virtually an HHG if
and only if it has
finite monodromy, by virtue of the (non) existence of some product
acylindrical action. However having finite monodromy is not a quasi-isometric
invariant, as the Leary - Minasyan group shows.

In \cite{bebf} the property (QT) is introduced, which applies to finitely
generated groups. Such a group has (QT) if it acts isometrically on a
finite product of quasitrees such that the orbit map is a quasi-isometric
embedding. It is shown there that many groups do have (QT). We show in
Corollary \ref{qt}, once again using the Leary - Minasyan group $L$,
that (QT) is not preserved under quasi-isometry. For this we do not
need knowledge of possible product acylindrical actions of generalised
Baumslag- Solitar groups $G$. We only need to know which elements will be
elliptic in all actions of $G$ on a hyperbolic space.

The author would like to acknowledge helpful conversations with
Mark Hagen and Sam Hughes, as well as the referee for useful suggestions
that improved the exposition (including pointing out the reference
\cite{kou}).

\section{Actions on hyperbolic spaces}

In this paper we consider groups $G$ acting
by isometries on a hyperbolic space (or later a finite product of hyperbolic
spaces). Here a hyperbolic space $X$ will always mean that $X$ is
a geodesic metric space
satisfying any of the equivalent definitions of $\delta$ - hyperbolicity.
Note that no further
conditions such as properness of the space will be assumed. 
As $X$ is hyperbolic we can look at the
action of $G$ by homeomorphisms
on the (Gromov) boundary $\partial X$ of $X$ to obtain the limit
set $\partial_G X$ which is a subset of $\partial X$. This subset is
$G$-invariant and we have $\partial_H X\subseteq \partial_G X$ if $H$
is a subgroup of $G$.

We have the standard classification of the individual elements
$g\in G$ as follows:\\
\begin{defn} \label{cyc}
(i) The element $g$ is {\bf elliptic} under the given action if
the subgroup $\langle g\rangle$ has bounded orbits. This happens
if and only if $\partial_{\langle g\rangle} X=\emptyset$.\\
(ii) The element $g$ is {\bf loxodromic} under the given action if
the subgroup $\langle g\rangle$ embeds quasi-isometrically in $X$
under the (or an) orbit map, namely there is
$c>0$ (which exists independently of $x\in X$) such that for all $n\in\Z$
we have $d(g^n(x),x)\geq |n|c$. This occurs if and only if
$\partial_{\langle g\rangle} X$ consists
of exactly 2 points $\{g^\pm\}$ for any $x\in X$ 
and this is the fixed point set of $g$ on $\partial X$.\\
(iii) The element $g$ is {\bf parabolic} exactly when it is not elliptic
or loxodromic, which occurs if and only if
$\partial_{\langle g\rangle} X$ consists
of exactly 1 point and again this is the fixed point set of $g$
on $\partial X$.
\end{defn}
Note that an element $g\in G$ is elliptic/loxodromic/parabolic if
and only if $g^n$ is (for some/all $n\in\Z\setminus \{0\}$) and if and
only if some conjugate of $g$ is.

Moving back now to arbitrary groups, if $G$ 
acts by isometries on an arbitrary hyperbolic space $X$ then we have the
Gromov classification dividing possible actions into five very different
classes (for these facts and related references, see 
\cite{abos} and \cite{ren}):\\
\hfill\\
(1) The action has {\bf bounded orbits}.\\
(2) The action is {\bf parabolic} (or horocyclic),
meaning that $\partial_G X$ has exactly
one point $p$. In this case the action can never be cobounded.\\
(3) The action is {\bf lineal}, meaning that
$\partial_G X=\{p,q\}$ has exactly 2 points (which in general can be swapped
or fixed pointwise by the action of $G$).
In this case there
will exist some loxodromic element in $G$ with limit set $\{p,q\}$.\\
(4) The action is {\bf quasi-parabolic} (or focal).
This says that the limit set
has at least 3 points, so is infinite, but there is some point
$p\in\partial_G X$ which is globally fixed by $G$. This implies that
$G$ contains a pair of loxodromic elements with limit sets
$\{p,q\}$ and $\{p,r\}$
for $p,q,r$ distinct points.\\
(5) The action is {\bf general}: the limit set is infinite and we have two
loxodromic elements with disjoint limit sets.

We will be interested in the question:
given a specific group $G$ and a particular element $g\in G$, can we find
an action of $G$ on some hyperbolic space $X$ where $g$ acts loxodromically?
If we first consider obstructions, the most obvious is if $g$ has finite
order or if $G$ is finitely generated and $g\in G$ is distorted in $G$.
As for sufficient conditions, if we have a homomorphism
$\theta: G\rightarrow \R$ with $\theta(g)\neq 0$ then we can realise
$\theta$ as an action of $G$ on $\R$ by translations in which $g$ acts
loxodromically. Indeed this also works for a homomorphism
$\theta:G\rightarrow Isom(\R)\cong\R\rtimes C_2$ if $\theta(g)$ has infinite
order, as here the elements swapping the ends of $\R$ will have order 2.
However, much more generally, this also works by \cite{abos} Proposition 4.9
if we have a homogeneous quasi-morphism $q:G\rightarrow\R$ with $q(g)\neq 0$,
whereupon $X$ is a quasiline.

We also have a type of converse which we will use later:
given an action of a locally compact group $G$ on a hyperbolic space $X$
where there is a
point $p\in\partial X$ on the boundary fixed by all of $G$, there exist
{\bf Busemann functions} on $G$. Here all groups with such
an action that we consider are countable and
discrete, hence locally compact.
In particular, see \cite{ccmt} where
it is shown that there exists a homogeneous quasi-morphism
$q:G\rightarrow \R$, the Busemann quasicharacter,
where the quasi-kernel $\{g\in G\,|\,q(g)=0\}$
consists exactly of those elements of $G$ which are not acting loxodromically
on $X$. In particular, although this results in $q$ being trivial in actions
of the first or second type, in the third (assuming the two limit points
are fixed pointwise) or the fourth types of actions we do obtain a non-trivial
homogeneous quasi-morphism on $G$. Moreover \cite{ccmt} also shows that
this Busemann quasicharacter is a genuine homomorphism if either $G$ is
amenable or $X$ is proper. Indeed the first point follows from the well
known fact that the only homogeneous quasi-morphisms on an amenable group
are the homomorphisms.

\section[Generalised
Baumslag - Solitar groups]{Generalised rank $n$ Baumslag - Solitar
  groups}

Let $G$ be the fundamental group of a finite graph of groups where all
edge and vertex groups are isomorphic to $\Z$. The classical
Baumslag - Solitar groups are obtained when the underlying
graph has a single vertex and one loop. For an arbitrary finite
graph, these groups go back at least to \cite{krpp} and at some later 
point they became known as generalised Baumslag - Solitar groups. 
However one can replace $\Z$ with $\Z^n$ for a fixed integer $n\geq 1$,
which are also sometimes known as generalised Baumslag - Solitar groups.
Therefore for clarity we introduce the following definition.
\begin{defn}
A {\bf     generalised rank $n$ Baumslag - Solitar group} ($GBS_n$ group)
$G$ is the fundamental group of a finite graph $\Gamma$
of groups where all vertex and
edge groups are isomorphic to $\Z^n$.
\end{defn}
This notation occurs in \cite{beekjgt}, where also the term $vGBS$ group is
used (with $v$ standing for variable) for when the rank of the various
vertex and edge groups $\Z^k$ need not be constant. However $vGBS$ groups
need not enjoy the finite valence property mentioned below and so we will not
consider those groups further.

\subsection{The modular homomorphism}

We can obtain a finite presentation for a $GBS_n$ group $G$ by taking each
vertex $v_i\in \Gamma$ and corresponding vertex group $V_i$ and fixing a free
basis $a_{i,1},\ldots ,a_{i,n}$ for $V_i$. We take a maximal tree $T_0$
in the finite graph $\Gamma$ defining $G$ and form the amalgamation of the
$V_i$s over the finite index edge subgroups. We then add a stable letter
$t_j$ to the generators for each of the $r$ (say) edges
$e_j$ of $\Gamma\setminus\{T_0\}$ and the corresponding relations where $t_j$
conjugates the inclusion of the edge group $E_j$ at one end of $e_j$ to the
inclusion of $E_j$ at the other end. Note this means that $G=N\rtimes F_r$
where $N$ is given by the normal closure of the vertex groups $V_i$ in $G$
and $F_r=\langle t_1,\ldots ,t_r\rangle$.

For $n=1$ a modular homomorphism from $G$ to $\Q^\times$ was considered in
\cite{lev}. This was generalised to arbitrary $n$ in \cite{dcvl} where
the homomorphism is now from $G$ to $\R^n\rtimes GL(n,\R)$. We present our
own more basic version here for arbitrary $n$ which has the advantage that
it works for the finite index subgroups of $G$ as well. It is equivalent
to the modular representation in \cite{cl} Section 3 but is a more concrete
version, making it easier to calculate in specific examples.

As a $GBS_n$ group is formed from a graph of groups, it
acts (by automorphisms without inverting edges) coboundedly
on a simplicial tree $T$. The crucial property here is that
this tree has finite valence because a subgroup of $\Z^n$ which is itself
isomorphic to $\Z^n$ must be a finite index subgroup. Now
given any $GBS_n$ group $G$, we first make arbitrary choices of a base
vertex $v_0$ in the tree $T$ on which $G$ acts, a finite index
subgroup $A\cong\Z^n$ of the stabiliser $Stab_G(v_0)\cong\Z^n$ and an ordered
basis $a_1,\ldots ,a_n$ of $A$. We refer to $A$ as our {\bf base
$\Z^n$-subgroup} for $G$. But
given any element $g\in G$, we have that
$A$ and $g^{-1}Ag$ are commensurable subgroups of $G$, that is their
intersection has finite index in both. This is because $g^{-1}Ag\leq_f
Stab_G(g^{-1}(v_0))=g^{-1}Stab(v_0)g$ and the finite valence of the tree $T$
implies that any two vertex stabilisers are commensurable so this holds
for $Stab_G(v_0)$ and $g^{-1}Stab(v_0)g$, hence also for $A$ and $g^{-1}Ag$.

Our modular homomorphism $\cal M$ which we now define will be
a homomorphism from $G$ to $GL(n,\Q)$ and it is easily checked that
changing any of these choices results in a homomorphism that is
conjugate in $GL(n,\Q)$ to $\cal M$. We first note that
for any finite index subgroup $H$ of $\Z^n$, there is $m\in\N$
(depending on $H$) such that $m\Z^n\leq_f H\leq_f \Z^n$ (here using additive
notation).
\begin{defn}
Let $G$ be a $GBS_n$ group for $n\geq 1$ with some base subgroup $A\cong\Z^n$. 
We define the {\bf modular homomorphism} ${\cal M}: G \rightarrow GL(n,\Q)$
in the following way. For any $g\in G$, we have just remarked that
(taking $\Z^n$ to be our base subgroup $A$ and $H$ to be $A\cap g^{-1}Ag$)
there is $m>0$ such that
\[mA=\langle a_1^m,\ldots ,a_n^m\rangle\leq_f A\cap g^{-1}Ag\leq_f A.\]

In particular for any $a=a_1^{l_1}\ldots a_n^{l_n}\in A$ with
$l_1,\ldots ,l_n\in \Z$, we have that $a^m\in g^{-1}Ag$ and so the element
$ga^mg^{-1}$ of $G$ is actually in $A$. This means that on taking $a$ to be
each of our basis elements $a_1,\ldots ,a_n$ in turn, we have uniquely
defined integer coefficients $g_{ij}$ with
\[ga_j^mg^{-1}=a_1^{g_{1j}}\ldots a_n^{g_{nj}}\]
and our definition of the modular map $\cal M$ is that it
sends $g$ to the matrix whose $i,j$th entry is
$g_{ij}/m\in \Q$.

The {\bf monodromy} of a $GBS_n$ group $G$ is the image ${\cal M}(G)$.
\end{defn}
Note that this definition of ${\cal M}(g)$
is independent of the value of $m$ taken
as if we replace $m$ with $m'$ for any appropriate $m'>0$ then
$(ga^mg^{-1})^{m'}$ is equal to $(ga^{m'}g^{-1})^m$. Moreover for any $g,h\in G$
we have ${\cal M}(g){\cal M}(h)={\cal M}(gh)$ so that $\cal M$ maps to
$GL(n,\Q)$ and is a homomorphism. Clearly the subgroup $A$ is in the
kernel of $\cal M$. We also note here that any
element $g$ of $G$ acting elliptically on the tree $T$ is sent to the
identity by $\cal M$ because $g$ lies in some vertex stabiliser
$Stab(v)\cong\Z^n$ and by finite valence of $T$ there will be a finite
index subgroup $B\leq_f A$ with $B\leq Stab(v)$, whereupon $gbg^{-1}=b$
for all $b\in B$. Thus $\cal M$ factors through the decomposition
of $G$ into $N\rtimes F_r$ above and so can also been thought of as a
homomorphism from $F_r$ to $GL(n,\Q)$. 
In particular if $r=0$ so that the underlying graph $\Gamma$ is actually
a finite tree then $\cal M$ is the trivial homomorphism. This can also
happen for $r>0$, for instance the HNN extension
$\langle a,b,t\,|\,[a,b],tat^{-1}=a,tbt^{-1}=b\rangle $ when $n=2$.

Given our modular homomorphism ${\cal M}:G\rightarrow GL(n,\Q)$, we can
restrict ${\mathcal M}$
to a subgroup $H$ of $G$. If $H\leq_f G$ then $H$ is also
a $GBS_n$ group because the restriction of the action of $G$ on the tree $T$
to $H$ is also cobounded, with edge and vertex subgroups which are finite
index subgroups of $\Z^n$, hence are all isomorphic to $\Z^n$ too. This
description of $H$ gives rise to its own modular homomorphism ${\cal M}_H$
but we can regard this, up to conjugacy, as the restriction of the modular
homomorphism $\cal M$ for $G$. This is because as $G$ and $H$ are acting
on the same tree $T$, we can first take the same base vertex $v_0$ in $T$.
Then since $H\leq_f G$, we have that $H\cap Stab_G(v_0)\leq_f Stab_G(v_0)$.
Thus we can choose our subgroup $A$ to be $H\cap Stab_G(v_0)$ both when
defining $\cal M$ and ${\cal M}_H$ and we also choose the same ordered basis 
for $A$ in both cases. We are now in the position that the definition of
${\cal M}_H$ is exactly the definition of $\cal M$ but just for elements
$h\in H$.

\subsection{The free abelianisation}

If $G$ is a finitely generated group and $G'$ is its commutator subgroup
then $G/G'$ is the abelianisation of $G$. It is a finitely generated
abelian group and so is of the form $\Z^k\oplus T$ where the torsion
subgroup $T$ is finite. Moreover every abelian quotient
of $G$ factors through $G/G'$.

Here we will consider the {\bf free abelianisation} $\overline{G}$
where we further quotient out by the torsion in the abelianisation
to obtain $\overline{G}=\Z^k$ for some $k$. This has the corresponding
universal property that any homomorphism from $G$ to a torsion free
abelian group factors through $\overline{G}$.

If we are given a finite presentation for $G$ with $m$ generators
then it is easy to
calculate $G/G'$ and $\overline{G}$ by abelianising these relations and
considering them as defining a subgroup $S$ of $\Z^m$ so that $G/G'$
is the quotient abelian group $\Z^m/S$. In fact the process is even easier
for $\overline{G}$ because of the lack of torsion: we can work over $\Q$
to get that the rank $k$ of $\overline{G}$ is the dimension of the
quotient space $\Q^m/R$, where $\Q^m$ is the vector space spanned by the
given generators for $G$ and $R$ is the subspace spanned by the relators
for $G$, once these relators have been abelianised and regarded as 
elements of $\Q^m$. This also says that an element $g\in G$ has infinite
order (equivalently is non trivial) in $\overline{G}$ if on expressing $g$
as a word in the generators and abelianising this word, the corresponding
$\Q^m$-vector is not in the subspace $R$.

This process works out especially well for a $GBS_n$ group $G$. As before,
we take our base vertex $v_0$, finite index subgroup $A$ of $Stab_G(v_0)$
and basis $a_1,\ldots ,a_n$ for $A$. On considering the 
$\Q$-vector space $W$ of dimension $n$
spanned by this basis, let us consider how this relates
to forming the group $G$ as the fundamental group of a finite graph of
$\Z^n$ groups and how it also relates to $\overline{G}$. Each time we
introduce a new vertex group $V_i$ and form its amalgamation with the
previous vertex groups over the appropriate edge group, we are giving an
identification of the $\Q$-vector space spanned by $V_i$ with our original
vector space $W$. Thus if there are no stable letters then
$\overline{G}=\Z^n$. However on taking
one of the stable letters $t_j$ with its edge running from the vertex
$v_j$ to the vertex $v_{j'}$, this introduces $n$ new relations in $G$
of the form
\[t_jx_1^{l_1}\ldots x_n^{l_n}t_j^{-1}=y_1^{m_1}\ldots y_n^{m_n}\]
where $x_1,\ldots ,x_n$ is a basis for
the vertex group $V_j$ and $y_1,\ldots ,y_n$ a basis for $V_{j'}$. Thus
in $\overline{G}$ we obtain the abelianised relation
$l_1x_1+\ldots +l_nx_n
=m_1y_1+\ldots +m_ny_n$, so that the corresponding relator can be expressed
using our identifications
above as an element of $W$ which is trivial
in $\overline{G}$. Thus the span of these $rn$ relators forms a subspace $R$
of $W\cong\Q^n$ and our free abelianisation $\overline{G}$ can be described
over $\Q$ as the quotient vector space $\Q^r\oplus (W/R)$, where the first
summand comes from the stable letters. Note that we can again work out
easily whether an element $g$ of $G$ is non trivial in $\overline{G}$ by
writing $g$ in terms of the generators obtained from the graph of groups
decomposition and abelianising. Indeed $g$ will be trivial
in $\overline{G}$ if and only if
each stable letter appears in $g$ with exponent sum 0 and such that
the resulting abelianisation of the word representing $g$, which will now lie
in $W$, also lies in $R$. Thus in particular we have from this discussion:
\begin{thm} \label{fab}
Suppose that $G$ is a $GBS_n$ group with base $\Z^n$-subgroup $A$.
Then $A$ embeds in
the free abelianisation $\overline{G}$ if and only if $R=\{0\}$. Moreover
this happens
if and only if the the monodromy ${\cal M}(G)$ is trivial
because the modular homomorphism is defined by what it does on the stable
letters.
\end{thm}

Note: from this, we see that the set of elements in $A$ which are trivial
(equivalently have finite order) in the free abelianisation of $G$ form
a subgroup of $A$. This is because these are the elements of $A$ which,
when considered as elements of $\Q^n$, lie in the subspace $R$ of $W$.

We can now present a dichotomy in the behaviour of $GBS_n$ groups (which
for $n=1$ is shown in \cite{lev} Proposition 2.6 and for general $n$ was stated
in \cite{cl} Section 3).
\begin{thm} \label{fnmn}
If $G$ is any $GBS_n$ group with finite monodromy then $G$
is virtually $\Z^n\times F_r$ for some $r\geq 0$.
\end{thm}
\begin{proof}
First drop down to a finite index subgroup $H$ of $G$ where $H$ has trivial
monodromy, for instance we could take $H$ to be the kernel of the
modular homomorphism ${\cal M}$.
As $H$ is the fundamental group of a finite graph of groups where
all of the finitely many edge and vertex groups are commensurable, we can
intersect them to get a subgroup $B$ of $H$ which can be used as a base
$\Z^n$-subgroup $H$. Now let us consider the presentation we obtain for $H$
from this graph of groups decomposition. Our finite generating set consists
of generators $g_i$ of the vertex groups along with the stable letters $t_j$.
Now any element $b\in B$ will also lie in any vertex group and so will commute
with every element $g_i$. On taking a stable letter $t$ which is obtained from
the edge joining the vertices $v_1$ and $v_2$ (possibly the same vertex) with
vertex groups $V_1,V_2$ say and respective edge inclusions $E_1\leq V_1$ and
$E_2\leq V_2$, we have that $tE_1t^{-1}=E_2$.

But as the monodromy is trivial, there is $M>0$
(depending on $b\in B$) such that
$b^M\in B\cap t^{-1}Bt$ and $tb^Mt^{-1}=b^M$. However $B$ lies in every edge and
vertex group, so that $b\in E_1$ and hence $tbt^{-1}\in E_2\cong \Z^n$. Thus
$tbt^{-1}$ must be an element in $E_2$ such that $(tbt^{-1})^M=b^M\in E_2$.
Clearly the element $b$ has this property as $B\leq E_2$ as well. Moreover
$M$th roots are unique in $\Z^n$, thus $tbt^{-1}=b$.

Hence we conclude that $B$ is normal and indeed central in $H$. As it lies in
(and is normal in) every vertex and edge group, we can consider $H/B$. This
group itself admits a graph of groups decomposition with the same underlying
finite graph, but with vertex groups $V_i/B$ and edge groups $E_j/B$. These are
all finite groups so $H/B$ is virtually free. Hence we can pull back a
finite index free
subgroup of $H/B$ to obtain a subgroup $L\leq_fH$ with $L/B\cong F_r$. But
as this quotient is free the extension splits, so there is a copy of $F_r$
in $L$ with $L=B\rtimes F_r$. As $B$ is central this is simply a direct
product, so that $G$ has the finite index subgroup $L\cong\Z^n\times F_r$.
\end{proof}

\subsection{Actions of $GBS_n$ groups on arbitrary hyperbolic spaces}
We can now use the above to see how a given $GBS_n$ group $G$ and its finite
index subgroups $H$ act on hyperbolic spaces. Certainly we have the action
of $G$ on its Bass - Serre tree where all elements of
the base $\Z^n$-subgroup $A$ act elliptically. However the
point is that if some element $a$ of $A$ is loxodromic when
$G$ acts on a hyperbolic space, the fact that the base subgroup
$A$ is commensurated in all of $G$ means
that the action must be very restricted.

\begin{thm} \label{act}
Suppose the monodromy of a $GBS_n$ group $G$ with given base $\Z^n$-subgroup
$A$ has infinite order. Then there is a non-trivial element $z\in A$
such that for any isometric action
of $G$ on a hyperbolic space, the element $z$ does not act
loxodromically.
\end{thm}
\begin{proof}
Consider any isometric action of $G$ on some hyperbolic space $X$ and
suppose that there is some element $a\in A$ which is acting loxodromically.  
We have the two limit points $p^\pm\in\partial X$ for the action of
$\langle a\rangle$ and
these are the only points in $\partial X$ fixed by any non-trivial
power of $a$. Now for any $g\in G$,
we have that $A$ and $g^{-1}Ag$ are commensurable subgroups.
As $a\in A$, we can find $j>0$ such that $a^j$ is in both of these
subgroups. In particular $ga^jg^{-1}\in A$ and it acts loxodromically
on $X$ because it is a conjugate in $G$ of the element $a^j$. But as
$A$ is abelian, this means $ga^jg^{-1}$ sends the fixed point set of
$a^j$ to itself and so swaps or fixes the two points $p^+$ and $p^-$.
However the case where $p^+$ and $p^-$ are swapped cannot occur: if so
then the loxodromic element $ga^jg^{-1}$ would have two other fixed
points $q^+,q^-$ say on $\partial X$. But then
the square of $ga^jg^{-1}$, which is also loxodromic,
would fix at least four distinct points
$p^\pm, q^\pm$ on $\partial X$, which loxodromic elements cannot do.
Thus for every $g\in G$
we have that $a^j$ and $ga^jg^{-1}$ both have the two fixed points
$p^+$ and $p^-$.
But the latter element actually has fixed
points $g(p^+)$ and $g(p^-)$ in the action of $G$ on $\partial X$,
so $\{p^+,p^-\}$ is preserved by every element
of $G$ and hence the action must be lineal of type (3).
Here we can have elements
of $G$ which swap the two points, but if so then we can avoid this by
dropping down to a subgroup of index 2.

Therefore let $G_0$ be the index 1 or 2 subgroup of $G$ where the action
preserves $p^+$ and $p^-$ pointwise. We can now take the Busemann
quasicharacter $q:G_0 \rightarrow \R$ (at $p^+$ say)
which is a homogeneous quasi-morphism
on $G_0$ that restricts to a genuine homomorphism $\theta =q|_{G_0\cap A}$
on $G_0\cap A$ (which is amenable). Note that $\theta$ is non-trivial
because $a\in G_0\cap A$ is acting loxodromically.

We will now show that we can ``lift'' $\theta$ to $G_0$, in that there
is some homomorphism $\Theta: G_0\rightarrow \R$ which restricts to
$\theta$ on $G_0\cap A$.
As $G$ is a $GBS_n$ group with a graph of groups decomposition giving us
a presentation for $G$, we can do the same for $G_0$ by restricting
the action of $G$ on the Bass - Serre tree $T$ to $G_0$ and then taking a
quotient. We already have $\Theta$
defined on $G_0\cap A$, which is a finite index subgroup of the vertex group
at the base vertex in the finite graph $\Gamma=G_0\backslash T$
and we can extend $\Theta$ to the group generated by all vertex groups
but without the relations from the stable letters. This is because all 
other defining
relations are each given by an isomorphism between finite index subgroups
of vertex groups, so we proceed inductively by extending $\Theta$ to
the image of the edge group in the new vertex group, then we can extend
over this vertex group.

However we must also consider the defining relations for $G_0$ coming
from the stable letters. These are all of the form
\[tx_1^{l_1}\ldots x_n^{l_n}t^{-1}=y_1^{m_1}\ldots y_n^{m_n}\]
where $t$ is one of these stable letters. We allow $\Theta$ to send the
stable letters anywhere but $h_1:=x_1^{l_1}\ldots x_n^{l_n}$ is an element
in the vertex group at one end of the edge defining $t$ and
$h_2:=y_1^{m_1}\ldots y_n^{m_n}$ is in the vertex group at the other end.
We must now show that $\Theta$ is still well defined when these relations
of the form $th_1t^{-1}=h_2$
are added. Now $h_1$ and $h_2$ both lie in vertex stabilisers so
by commensurability there will be $M>0$ such that $h_1^M$ and $h_2^M$ are both
in our base $\Z^n$-subgroup $G_0\cap A$.
But $t,h_1,h_2$ all lie in
$G_0$ which is the domain of the homogeneous quasi-morphism $q$, thus we
have $q(th_1^Mt^{-1})=q(h_2^M)$. Now homogeneous quasi-morphisms are
invariant under conjugation (because $|q(yxy^{-1})-q(x)|$ is bounded
independently of $x$, so we can replace $x$ with $x^m$ and let $m$ tend
to infinity), thus this becomes $q(h_1^M)=q(h_2^M)$.
But $q(h_i^M)=\theta(h_i^M)$ for $i=1,2$ as $h_1^M,h_2^M$ are in the domain
of $\theta$ and we have $\Theta(h_i)=(1/M)\Theta(h_i^M)$ by definition of
$\Theta$ above. Also $\Theta(th_1t^{-1})=\Theta(h_1)$ as we are mapping
to $\R$, so this is also equal to
\[(1/M)\Theta(h_1^M)=
(1/M)\theta(h_1^M)=(1/M)\theta(h_2^M)=(1/M)\Theta(h_2^M)=\Theta(h_2)\]
so $\Theta$ is indeed well defined.

We can now finish the proof. The subgroup $G_0$ obtained above
has index 1 or 2 in $G$ and there are only finitely many index 2 subgroups
in $G$ as it is finitely generated. Let $G_2$ be the intersection of all
of these index 2 subgroups, which will have finite index in $G$. As $G$
has infinite monodromy, the same is true for $G_2$ where we can take our
base $\Z^n$-subgroup to be $G_2\cap A$. Thus by
Theorem \ref{fab} applied to $G_2$ with base $\Z^n$-subgroup $G_2\cap A$,
there is some non-identity
$z\in G_2\cap A$ which is trivial in the free abelianisation of $G_2$.
Let us now suppose that $G$ does have an action on some hyperbolic space
$X$ in which $z$ is loxodromic, so we can run through the above
argument with $z$ equal to $a$. But then we will obtain a homomorphism
$\Theta$ from some index 2 subgroup of $G$ to $\R$ with $\Theta(z)\neq 0$
and this subgroup will contain $G_2$. Thus we can restrict $\Theta$
to $G_2$ which is a contradiction because $z$ is trivial in the free
abelianisation of $G_2$.
\end{proof}

Note that this theorem can be used to complement Theorem \ref{fnmn} in that
if a $GBS_n$ group $G$ has infinite monodromy then it can have no finite
index subgroup $H$ which is isomorphic to a 
a direct product of a free group and copies of $\Z$. This is because
$H$ would have infinite monodromy as well, thus by Theorem \ref{act}
there would be a non-trivial element of $H$ which cannot be loxodromic
in any action of $H$ on a hyperbolic space. But if $H$ has the above form
then we can use actions on trees to make any given element loxodromic.

This at least deals with a potential ambiguity which we have previously
glossed over: that $G$ could have different decompositions as a generalised
Baumslag - Solitar group. For instance $\Z^{n+1}$ is both a $GBS_{n+1}$
group in a trivial way where the underlying graph is a single vertex
and also a $GBS_n$ group $\Z^n\times \Z$ with graph a vertex and an
edge, where the stable letter acts trivially by conjugation and where
we have the freedom to take any primitive element as the stable letter.
Here the monodromy is trivial in all cases, but if we take the fundamental
group of the Nil torus bundle
\[\langle t,a,b\,|\, [a,b]=e, tat^{-1}=ba, tbt^{-1}=b\rangle \]
which is of the form $\Z^2\rtimes\Z$, we see that we have a decomposition
with base group $\langle a,b\rangle$ and stable letter $t$, or base
group $\langle t,b\rangle$ and stable letter $a$. The respective
modular homomorphisms of these two decompositions are different
but both have infinite monondromy.

The above tells us that whether the monodromy is finite or infinite (which is
our main concern in this paper) is an invariant only of the group $G$, not of
how it decomposes as a generalised Baumslag - Solitar group.
In fact in nearly all cases the modular homomorphism itself and thus the
monodromy (up to conjugation in $GL(n,\Q)$) is well defined. We state
this as:
\begin{prop}
  Take any two decompositions of the same group $G$ as a generalised
  Baumslag - Solitar group. If neither of the two actions on the
  corresponding Bass - Serre trees have an invariant line or point,
  the modular homomorphisms corresponding to the two decompositions are
  the same.
\end{prop}
\begin{proof} Here we adapt
the result in \cite{fordef} Corollary 6.10 which achieves this for $n=1$.
Suppose that
a group $G$ has two decompositions as a generalised Baumslag - Solitar
group with the first of rank $n$ say. By considering cohomological
dimension, the second is also of rank $n$ unless one decomposition is
trivial (in which case $G$ is abelian and so the modular homomorphism
obtained from any possible decomposition will be trivial too).
Hence we have two actions of $G$ on trees $T_1$ and $T_2$
say and hence two partitions $\{E_1,H_1\}$ and $\{E_2,H_2\}$
of $G$ into elliptic and hyperbolic elements
according to these actions. But if these partitions are the same then
the resulting modular homomorphisms are the same.
This is because elliptic
elements are sent to the identity and the image of a hyperbolic element
$g$ is determined by how it conjugates elements in any base $\Z^n$-subgroup.
But under any action of $G$ on a tree (by automorphisms without
inverting edges), a finitely generated purely elliptic subgroup will
fix some vertex, so we can use the same base $\Z^n$-subgroup in either
decomposition.

Thus we are done if we have a characterisation of the hyperbolic and
elliptic elements of $G$ which does not depend on the particular action.
In \cite{fordef} this is achieved (in all but these exceptional cases)
for $n=1$ by showing that the elliptics are the elements of $G$
which are commensurable with all their conjugates. Here we can instead argue:
suppose that $g\in G$ lies in $E_2\setminus E_1$ so that it fixes a vertex
$v_2$ when $G$ acts on $T_2$ but which is hyperbolic when $G$ acts on $T_1$.
The action of $G$ on $T_2$ gives rise to a generalised Baumslag - Solitar
decomposition of $G$ where we can take the base group to be $Stab_G(v_2)$.
Thus on taking $g$ equal to $a$ in the proof of Theorem \ref{act} where we
use the action of $G$ on $T_1$ with $g$ loxodromic, we conclude by
the same argument that $G$ fixes setwise the axis of $g$. This means that
this $GBS_n$ decomposition of $G$ gives rise to an invariant line
when $G$ acts on the Bass - Serre tree which is ruled out by hypothesis.
(In fact a brief check reveals that in the exceptional case,
$G$ can only be equal to
$\Z^n$, $\Z^n\rtimes_\alpha\Z$ for $\alpha$ some automorphism of $\Z^n$,
or $\Z^n*_C\Z^n$ where $C$ has index 2 in both copies of $\Z^n$.) Thus
we see that $g$ was not loxodromic
and so $E_2\subseteq E_1$. We can now swap the actions and argue again to
conclude that $E_1=E_2$ and $H_1=H_2$.
\end{proof}

We finish this section by noting though that 
even though the modular homomorphism is well defined (away from these
exceptional cases), a group might
have many decompositions as a generalised Baumslag - Solitar group
which are not obviously related. For instance the isomorphism problem
is open just amongst $GBS_1$ groups (see \cite{mntil} for some recent
progress on this question). Moreover \cite{arm} exhibits $GBS_4$ groups with
unsolvable conjugacy problem.

\section[Actions on products of hyperbolic spaces]{Acylindrical actions on
 products of hyperbolic spaces}

\subsection{Product acylindrical actions}

Given a group $G$ acting by isometries on a metric space $X$, a well known
definition is that of $G$ acting {\bf acylindrically}: that is, given any
$\epsilon\geq 0$ we have $N,R$ such that if $x,y\in X$ are two points which
are at least distance $R$ apart then the set of group elements moving 
both $x$ and $y$ by at most $\epsilon$ has cardinality at most $N$. This
definition is generally used when $X$ is a hyperbolic metric space
whereupon it gives rise to the concept of a group being {\bf acylindrically
  hyperbolic}. This is where there exists an acylindrical action which is of
type (5) in the earlier list (in fact actions of types (2) and (4) can
never be acylindrical on hyperbolic spaces)
and which implies a number of consequences, for instance
such a group must be SQ-universal.

Observe that the definition of an acylindrical action makes sense for any
group action by isometries on any metric space. However such a concept is
not useful in this generality. First of all if $X$ is bounded then any
action is trivially acylindrical, so any suitable notion needs to avoid
this case.
(Indeed even if $X$ is unbounded then an acylindrical action
can still have bounded orbits, but not all actions with bounded orbits
need be acylindrical.) But even this is problematic because any
geometric action on any metric space is uniformly metrically proper, which
in turn implies that the action is acylindrical. Therefore any finitely
generated group acts acylindrically on its own Cayley graph and consequently
there is no chance that an acylindrical action automatically implies
any group theoretic consequences, even if we could agree on what a
suitable acylindrical action meant in this context.

But one option is to restrict the metric space $X$ to being hyperbolic-like
or to have non-positive curvature in some sense, whilst making it more general
than just a hyperbolic space. In this section we look at what happens when
we allow $X$ to be a finite product of hyperbolic spaces rather than just
one. This will allow us to create obstructions to $GBS_n$ groups having
such an action, which we will then compare to the class of hierarchically
hyperbolic groups in the next section. Suppose we have a product of $r$
metric spaces $P=X_1\times \ldots \times X_r$ (where $P$ is equipped with
the $\ell_1$ product metric) and an isometric action of a group $G$ on $P$.
Note that $Isom(X_1)\times \ldots \times Isom(X_r)$ is naturally a subgroup
of $Isom(X)$ using the diagonal action.
We say that $G$ acts on $P$ {\bf preserving factors} if the image of this
action lies inside $Isom(X_1)\times\ldots\times Isom(X_r)$, whereupon
we can think of any element $g\in G$ as having an expression $(g_1,\ldots ,g_r)$
with $g_i$ an isometry of $Isom(X_i)$. We now introduce our main definition
of this section.
\begin{defn} If a group $G$ acts isometrically on $P=X_1\times \ldots
\times X_r$ where each $X_i$ is a hyperbolic space then we say that
the action is {\bf product acylindrical} if the action on $P$ is acylindrical,
preserves the factors of $P$ and such that there is an element
$g=(g_1,\ldots , g_r)$ in $G$ where some $g_i$ acts as a loxodromic element
on the space $X_i$.
\end{defn}

Note that if $r=1$ then, as opposed to the standard definition of an
acylindrically hyperbolic group, we allow actions of type (3) (whereupon
our group will be virtually cyclic)
with an infinite order element acting loxodromically
as we will not need to treat this as a special case. However in common
with the standard definition, we rule out any action with bounded orbits.

We make some points about the above definition. First it need
not be the case that if $G$ acts acylindrically on on a space $X$ and
arbitrarily on another space $Y$ then the product action on $X\times Y$
is acylindrical. (One could take any acylindrical action of some group
$G$ on $X$ where a
point has an infinite stabiliser and then set $Y=\mathbb R$ with $G$
acting as the identity on $Y$.) However if $G$ acts uniformly metrically
properly on $X$ and arbitrarily on $Y$ then the product action on
$X\times Y$ will also be uniformly metrically proper and hence will be
an acylindrical action. Thus if $G$ does have a product acylindrical action
then we can say nothing in general about the action on any individual
factor. Moreover we do obtain groups which have a product acylindrical
action but which are not themselves acylindrically hyperbolic, for
instance $F_2\times F_2$ or Burger - Mozes - Wise groups acting
geometrically on a product of two trees. As these last groups can
be virtually simple, we also note that we do not have any result
for product acylindrical actions of the form: a group with 
such an action is SQ-universal
or virtually cyclic which we do have for unbounded acylindrical actions on a
single hyperbolic space.

In Definition \ref{cyc} we saw a division of isometries into
a trichotomy of elliptic/loxodromic/parabolic elements when the space
is hyperbolic. Note that we still
have this trichotomy for a group acting on an arbitrary metric space $X$:
elliptic elements have bounded orbits, loxodromic elements have orbits
which quasi-isometrically embed in $X$ and everything else is a parabolic
element. In general we might not have a nice description of how these
isometries act on the boundary of $X$ (indeed there might not even be
a suitable boundary of $X$).
However we do have a simple way of classifying an element of a group
acting on a finite product of metric spaces preserving factors
if we know how this element is behaving on each of the factors.
\begin{lem} \label{cycp}
Let $G$ act on the product of metric spaces $P=X_1\times \ldots\times X_r$
by isometries preserving factors and take $g=(g_1,\ldots ,g_r)\in G$ where
$g_i$ acts isometrically on $X_i$. Then:\\
(i) $g$ acts elliptically on $P$ if and only if $g_i$ acts elliptically
on $X_i$ for each $1\leq i\leq r$.\\
(ii) $g$ acts loxodromically on $P$ if and only if there is some $i$
where $g_i$ acts loxodromically on $X_i$.\\
(iii) $g$ acts parabolically on $P$ if and only if no $g_i$ acts
loxodromically on $X_i$ but some $g_i$ acts parabolically on $X_i$.
\end{lem}
\begin{proof}
The first case is straightforward to establish in both directions, just by
using the fact
that if in the product space $P=X_1\times \ldots \times X_r$ we have
points $x=(x_1,\ldots ,x_r)$ and $y=(y_1,\ldots ,y_r)$ then
$d_{X_i}(x_i,y_i)\leq d_P(x,y)$. The same is true for the reverse implication
of the second case.

Thus let us now suppose without loss of generality that $g_1$ acts
parabolically on $X_1$ and $g_i$ does not act loxodromically on $X_i$
for any $1\leq i\leq r$. Then
the orbit of any  point $x_1\in X_1$ under $\langle g_1\rangle$ will be
unbounded so the orbit of any $x\in P$ under $\langle g\rangle$ will be too.
Thus in order to establish the forward direction of (ii) and the reverse
direction of (iii), 
we just need to show that $g$ is not loxodromic, for which we can
use stable translation length. In particular for any point
$x=(x_1,\ldots ,x_r)\in P$, we have
$d_i(g_i^m(x_i),x_i)/m$ tending to zero as $m$ tends to
infinity because no $g_i$ acts loxodromically. But then
by adding and using inequalities we have that
$d_P(g^m(x),x)/m$ tends to zero too, thus we cannot have $c>0$ and
$\epsilon\geq 0$ with $d_P(g^m(x),x)\geq cm-\epsilon$ for all $m\in\N$
and so $g$ does not act loxodromically on $P$. Finally if $g$ acts
parabolically on $P$ then we cannot be in Cases (i) or (ii) by what we
have already shown.
\end{proof}

It was shown in \cite{bow} Lemma 2.2
that if $G$ acts acylindrically on a hyperbolic
space $X$ then no element can act parabolically.  For a general metric
space $X$, we can certainly have acylindrical actions with elements
acting parabolically. For instance take any finitely generated group $G$
with a distorted infinite cyclic subgroup $\langle g\rangle$. Then $G$
acts geometrically and hence acylindrically on its own Cayley graph,
but the action of $g$ here will be parabolic.

We now show however that we cannot have parabolic elements in a
product acylindrical action.
\begin{thm} \label{nop}
Suppose that a group $G$ acts acylindrically and preserving factors
on the product $P$ of hyperbolic spaces $X_1\times \ldots \times X_r$.
Then no element of $G$ acts as a parabolic element on $P$.
\end{thm}
\begin{proof}
Suppose otherwise, so that by Lemma \ref{cycp}
we have $g=(g_1,\ldots ,g_r)$ acting on $P$
where without loss of generality 
the action of $g_1$ on $X_1$ is parabolic and none of the actions of
$g_i$ on $X_i$ are loxodromic for $1\leq i\leq r$. We now invoke
\cite{kou} Proposition 3.2. This states that if $X$ is a geodesic
$\delta$-hyperbolic space and $S$ is any finite set of isometries of
$X$ such that neither $S$ nor $S^2$ contains a loxodromic element, then
the joint minimum displacement $L(S)$ is at most $100\delta$.
Here $L(S)$ is defined to be the
infimum over all points $x\in X$ of the joint displacement
$\mbox{max}_{s\in S} d(x,s(x))$.

For each $i$ we apply this result to the $\delta_i$-hyperbolic space $X_i$.
We have no loxodromic elements in $\langle g_i\rangle$, thus for any finite
subset $S$ of $\langle g_i\rangle$ there is some point $x_i\in X_i$ which
is moved by at most distance $101\delta_i$ by any $s\in S$. In particular
for any $N>0$, this applies to the set $S_N=\{g_i,\ldots ,g_i^N\}$. We will
label the point obtained above when this
result is applied with $S=S_N$ as $x_i^{(N)}\in X_i$.

We now show that this action of $\langle g\rangle$ on $P$ is not acylindrical
and hence nor is the action of $G$ on $P$. Set $\epsilon$ to be $101\delta r$
where $\delta=\mbox{max}(\delta_1,\ldots ,\delta_r)$
and suppose we are given $R$ and $N$. Now as $g_1$ does act parabolically
on $X_1$, the orbit of $x_1^{(N)}$ under $\langle g_1\rangle$
is not bounded and so we can find $g_1^K$
where, on setting $y_1^{(N)}:=g_1^K(x_1^{(N)})$ and then defining
$y_i^{(N)}:=g_i^K(x_i^{(N)})$ for the same $K$, we have
$d_1(x_1^{(N)},y_1^{(N)})\geq R$ where $d_i$ is the distance in $X_i$.

Now for all $1\leq i\leq r$ and $1\leq j\leq N$ we have that
$d_i(g_i^j(x_i^{(N)}),x_i^{(N)})\leq 101\delta$ and also
\[d_i(g_i^j(y_i^{(N)}),y_i^{(N)})
  =d_i(g_i^{K+j}(x_i^{(N)}),g_i^K(x_i^{(N)}))=
  d_i(g_i^j(x_i^{(N)}),x_i^{(N)})
\]
which therefore is at most $101\delta$ as well.
Thus if we take the two points $x=(x_1^{(N)},\ldots ,x_r^{(N)})$ 
and $y=(y_1^{(N)},\ldots ,y_r^{(N)})$ of $P$ then we have
$d(x,y)\geq d_1(x_1^{(N)},y_1^{(N)})\geq R$ but the $N$ distinct elements $g^j$
for $1\leq j\leq N$ satisfy
\[d(g^j(x),x)=d_1(g_1^j(x_1^{(N)}),x_1^{(N)})+\ldots +
  d_r(g_r^j(x_r^{(N)}),x_r^{(N)})=d(g^j(y),y)\leq 101\delta r\]
so that each of these elements moves both $x$ and $y$ by at most $\epsilon$.
\end{proof}

\subsection{Acylindrical actions of generalised Baumslag-Solitar groups}

A generalised Baumslag-Solitar group can never be acylindrically
hyperbolic. This can be seen by using \cite{osmn} Theorem 3.7,
which states that if $G$ is acylindrically hyperbolic and has a subgroup
$H$ where $H\cap gHg^{-1}$ is infinite for all $g\in G$ then $H$ must
itself be acylindrically hyperbolic. For a $GBS_n$ group $G$ we can of course
take $H$ to be a base $\Z^n$-subgroup for a contradiction.

However certainly there are $GBS_n$ groups which possess a product acylindrical
action, for instance $\Z^n\times F_r$ has an obvious geometric (and hence
acylindrical) action on $\R\times\ldots \times \R\times T_{2r}$ where
$T_d$ is the regular tree of degree $d$. We are interested in when a
$GBS_n$ group $G$ has a finite index subgroup $H$ possessing a product
acylindrical action. (Note that even for acylindrical hyperbolicity, it is not
known whether $H$ having this property and $H\leq_f G$ implies that $G$ has
this property too. See \cite{osmnc} which is a correction to \cite{osmn}
Lemma 3.9.) However if we consider the property of virtually having
a proper acylindrical action then we can now give a complete answer in the
case of $GBS_n$ groups.

\begin{thm} \label {pacg}
For any $n\geq 1$, a $GBS_n$ group $G$ has a finite index subgroup $H$
possessing a product acylindrical action if and only if the monodromy
of $G$ is finite.
\end{thm}
\begin{proof}
First if the monodromy of $G$ is finite then by Theorem \ref{fnmn}
we have $H\leq_f G$ with $H$ of the form $\Z^n\times F_r$ and this
has a geometric, hence acylindrical, action on the product of $n+1$
hyperbolic spaces which preserves factors.

Now say that the monodromy of $G$ is infinite. Then so will the
monodromy of $H$ for any finite index subgroup $H$, hence we just need to
rule out that $G$ has a product acylindrical action.

We thus suppose that $G$ acts on the product $P$ of hyperbolic spaces
$X_1\times\ldots \times X_r$ by isometries, and that this action preserves
factors and is acylindrical. By Theorem \ref{act} we have an infinite order
element $z$ of $G$ which lies in our base $\Z^n$-subgroup $A$ and
which cannot be loxodromic in any action of $G$ on a
hyperbolic space. Hence on splitting $z$ into its component parts
$(z_1,\ldots ,z_r)$ with each $z_i$ acting as an isometry of the hyperbolic
space $X_i$, we have that $z_i$ must act parabolically or elliptically
on $X_i$. By Lemma \ref{cycp} $z$ must be a parabolic or elliptic element
in the action of $G$ on $P$. But by Theorem \ref{nop} there are no parabolic
elements if we have a product acylindrical action of $G$.

Thus $z$ must be an elliptic element. Whilst we can certainly have
acylindrical actions with elliptic infinite order elements, or indeed
with every element acting elliptically, we are not in the latter case
here because in the definition of a product acylindrical
action we must have some other element
of $G$ which is acting loxodromically.

Hence let $g\in G$ be this element, with $g=(g_1,\ldots ,g_r)$ and without
loss of generality $g_1$ acts loxodromically on $X_1$. Let us set $G_2$
to be the intersection of the index 2 subgroups of $G$ as before and
note that by the
proof of Theorem \ref{act}, we can take for $z$ any non-identity element
of $A\cap G_2$ which is trivial in the free abelianisation of $G_2$.
Consider the subset $Z$ of $A\cap G_2$ consisting of
elements with this property along with the identity
and recall that this forms a subgroup of $A\cap G_2$ and so it is a
finitely generated free abelian group. It is also infinite (for instance it
contains all powers of a given $z\in Z$).
Now for the element $g$ as above, we can replace it by a power (which we
will continue to call $g$) which lies in $G_2$.

We will show that this action of $G$ on $P$ cannot in fact be acylindrical.
Pick any point $x_0\in P$ and let $D$ be an upper bound for the set
$\{d_P(z(x_0),x_0)\,|\,z\in Z\}$. This exists because every element
$z$ is acting
elliptically and so the orbit of $x$ under $Z$ is bounded, because
$Z$ is finitely generated and abelian.

Set $\epsilon=D$ and suppose we are given any $R>0$. As $g$ acts
loxodromically, there will exist $K>0$ such that $d_P(g^K(x_0),x_0)\geq R$
and set $y_R:=g^K(x_0)$. Now by definition of the modular homomorphism,
there will be an integer $m>0$ (depending on $R$) such that $g^Kz^mg^{-K}$
is in our base subgroup $A$ and hence in $G_2\cap A$ as $g^K\in G_2$.
But note that if $z^m$ is trivial in the free abelianisation of $G_2$
then so is $g^Kz^mg^{-K}$ as $g^K\in G_2$, so $z^m$ and $g^Kz^mg^{-K}$ will
have the same image in the free abelianisation of $G_2$. Notice this
also works for powers $z^{im}$ and $g^Kz^{im}g^{-K}$ for any $i>0$.

Thus we have infinitely many elements $\{g^Kz^{im}g^{-K}\,|\,i\in \N\}$
(which are distinct because $z^{im}$ are) lying in $Z$ and so they each
move $x_0$ by at most a distance $D$ in $P$. But clearly we also have
\[d_P(g^Kz^{im}g^{-K}(y_R),y_R)=d_P(g^Kz^{im}(x_0),g^K(x_0))=d_P(z^{im}(x_0),x_0)\]
which is at most $D$, and so this infinite set of elements moves both
$x_0$ and $y_R$ a distance at most $\epsilon$, but $d_P(x_0,Y_R)\geq R$.
Thus $G$ is not acting acylindrically on $P$.
\end{proof}

\section[HHGs and Generalised Baumslag - Solitar groups]{Hierarchically
 hyperbolic groups and Generalised Baumslag - Solitar groups}

\subsection{Hierarchically hyperbolic groups}

The notion of a hierarchically hyperbolic space (HHS) was introduced in
\cite{bhs1} as a way of generalising hyperbolic spaces to include
mapping class groups and many CAT(0) cube complexes. We do not give a
definition here but roughly speaking an HHS is a quasigeodesic metric
space $X$ together with a structure given in terms of
projections $\pi_i:X \rightarrow U_i$ 
to a family (infinite in general)
of hyperbolic spaces $\{U_i\,|\,i\in I\}$ called the domains (these
spaces need not be proper in general and they can also be bounded).
A hierarchically hyperbolic group (HHG) $G$ is not
merely a finitely generated
group quasi-isometric to a HHS but one where there is
a HHS $X$ where the group $G$ acts (or rather quasi-acts) geometrically
on $X$ and permutes the family of hyperbolic spaces by isometries.
This family also has a nesting and an orthogonality relation.
If $G$ is a HHG
then we can take $X$ to be the Cayley graph of $G$ with respect to a finite
generating set along with the usual action of $G$ on itself by left
multiplication (thus without loss of generality we do have a genuine action 
of $G$ on $X$ rather than a quasi-action). We can also assume without loss
of generality that the image $\pi_i(G)$ is coarsely dense in $U_i$ and that
this is uniform over $i\in I$.

At this point we might wonder what properties are possessed by HHGs.
We have:\\
\hfill\\ (i) Hyperbolic groups are HHGs (this is seen by taking the
family of hyperbolic spaces to be a single hyperbolic space).\\
\hfill\\
(ii) If $G$ is an HHG and $H\leq_f G$ then $H$ is an HHG
(\cite{abs} Lemma 2.25).\\
\hfill\\
(iii) If $G_1$ and $G_2$ are HHGs then so is $G_1\times G_2$
(\cite{bhs2} Corollary 8.28).

We might also wonder about obstructions for a given group $G$ to
be an HHG. Here we have:\\
\hfill\\
(1) Any HHG $G$ is finitely presented and has quadratic isoperimetric
inequality (\cite{bhs2} Corollary 7.5).\\
\hfill\\
(2) Given any finitely generated subgroup $H$ of an HHG $G$, either $H$ is
virtually abelian or $F_2\leq H$ (\cite{dhs} Theorem 9.15 where the
condition of being finitely generated is not used, but see also the
correction in \cite{dhsc}).\\
\hfill\\
(3) Every infinite order element $g\in G$ is undistorted in $G$ with
respect to word length of a finite generating set (see \cite{dhs}
Theorem 7.1 but again see also the correction in \cite{dhsc}).

Note that for each ordered pair of these three statements, there exists
a group satisfying the first but not the second.

More recently another obstruction was found in \cite{petspr}.
We take the following result from Remark 4.9 of that paper.
\begin{thm} \label{hhgp}
If $G$ is an infinite  HHG then there exist finitely many unbounded
hyperbolic spaces $X_1,\ldots ,X_r$ for $r\geq 1$
and an isometric action of
some finite index subgroup $H$ of $G$ on the
product space $P=X_1\times \ldots \times X_r$ (with the $l_1$ metric) 
such that $H$ preserves each factor $X_i$ and $H$ acts acylindrically
on the product $P$. Moreover there exists a loxodromic element in the
action of $H$ on $X_i$, thus this action of $H$ is product acylindrical.
\end{thm}
\begin{proof}
Theorem 3.2 of \cite{petspr} shows that for any HHG $G$ there is a finite
$G$-invariant set ${\cal W} =\{W_1,\ldots ,W_r\}$ of unbounded domains which
are pairwise orthogonal and such that any unbounded domain $U_i$ is nested
in one of these $W_j$.
In the case when $G$ is finite
$\cal W$ is necessarily empty as all domains will be bounded.
Conversely an HHS with all domains bounded is itself bounded, so whenever
$G$ is infinite there will be some unbounded domain and therefore
$\cal W$ is non empty.

Consequently there is a finite index subgroup $H$
of $G$ with $H(W_j)=W_j$ for all $1\leq j\leq r$ and $H$ acts by
isometries on each domain $W_j$.
Moreover the action of $H$ on each $W_j$ is
cobounded. This is because we can assume for any domain $U_i$ and any $g\in G,
x\in X(=G)$ that $g(\pi_i(x))=\pi_i(gx)$ by \cite{dhsc} Remark 2.1.
But we mentioned above that $\pi_i(G)=G(\pi_i(id))$
is coarsely dense and therefore so
is $\pi_i(H)=H(\pi_i(id))$ as $H$ has finite index in $G$.
However any cobounded action of a group on an unbounded
hyperbolic space must contain a loxodromic element, as mentioned above
when we listed the five types of action.

We can now get $H$ to act on the product $W_1\times \ldots \times W_r$
with the $\ell_1$ metric of these unbounded domains using the diagonal
action. This action clearly preserves factors and we have said that
it will contain a loxodromic element. But as pointed out in \cite{petspr}
Remark 4.9, the proof in \cite{bhs1} Theorem 14.3 that $G$ acts
acylindrically on $S$ in the case where $\cal W$ consists of the single
unbounded domain $S$ applies equally to the above action of $H$ on
$W_1\times\ldots\times W_r$ which preserves factors and contains a
loxodromic element in the action, thus this action of
$H$ is product acylindrical.
\end{proof}

\subsection{Generalised Baumslag - Solitar groups and  HHGs}

We can now use the results above for our main application. Our initial
question might be which generalised Baumslag - Solitar groups are
HHGs. However it is possible for a group $G$ to have a finite index subgroup
which is an HHG but for $G$ not to be. This was shown in \cite{petspr}
Corollary 4.5 by using Theorem 3.2 in that paper which we have already quoted
above. Specifically they show that the (orientation preserving)
$(3,3,3)$ triangle group is not an HHG but of course it has the finite index
subgroup $\Z^2$ which is. A group $G$ which is virtually an HHG will
still have good group theoretic and geometric properties as these will
be inherited from the finite index subgroup, so we now give a complete
answer to which generalised Baumslag - Solitar groups are virtually HHGs.
\begin{co} \label{nthhg}
If $G$ is a $GBS_n$ group then $G$ is virtually a HHG if
and only if $G$ has finite monodromy.
\end{co}
\begin{proof}
If $G$ has finite monodromy then $G$ is virtually $\Z^n\times F_r$
by Theorem \ref{fnmn} and the latter group is an HHG as it is a direct product
of HHGs.

Now suppose the $GBS_n$ group $G$ has a finite index subgroup $L$
which is a HHG. By Theorem \ref{hhgp} there is some finite index subgroup
$H$ of $L$ which has a product acylindrical action (unless $L$ is finite,
in which case $G$ was not a generalised Baumslag - Solitar group).
But $H$ is then a finite index subgroup of $G$ too and so
by Theorem \ref{pacg} we have that the monodromy of $G$ must be finite.
\end{proof}

The result above in \cite{petspr} that the $(3,3,3)$ triangle group is not
a HHG even though it is virtually $\Z^2$ established that being a HHG is
not a quasi-isometry invariant and indeed not even a commensurability
invariant. However whenever we have a property $\cal P$ of abstract groups
which is invariant under taking finite index subgroups, as is the case for
being a HHG, we can recover a commensurability invariant by using the
property of being virtually $\cal P$. Thus although \cite{petspr} tells us
that being a HHG is not a commensurability invariant and thus not a
quasi-isometric invariant, being virtually a HHG is certainly a
commensurability invariant so the obvious question now is whether it is
preserved under quasi-isometry. By our results here along with famous
recent examples of Leary - Minasyan, we can show the answer is no.
The next theorem is an adaptation of \cite{lrmn} Theorem 7.2 to $GBS_n$
groups.
\begin{thm} \label{latt}
  Suppose that $G$ is a $GBS_n$ group with monodromy
  ${\mathcal M}\leq GL(n,\R)$  which is conjugate in $GL(n,\R)$ to a subgroup
  of the orthogonal group $O(n)$. Then $G$ is a CAT(0) group acting properly,
  freely and cocompactly on $T\times \R^n$, where $T$ is a locally
  finite tree. Conversely if $G$ is virtually a CAT(0) group then
 its monodromy can be conjugated into $O(n)$. 
\end{thm}
\begin{proof} We take for $T$ the Bass - Serre tree of the decomposition
  of $G$ as a $GBS_n$ group, so that $G$ acts on $T$ by tree automorphisms.
  We also pick some base vertex $v_0$ of $T$
  and we set our base $\Z^n$-subgroup $A$ to be the intersection of all
  vertex and edge subgroups obtained from this decomposition, which will
  have finite index in each of these subgroups, along with a
  free generating set $\langle a_1,\ldots ,a_n\rangle$ for $A$. We will now
  define a homomorphism $\theta$ from $G$ to $Isom(\R^n)$
  and then let $G$ act on $T\times \R^n$ via the
  diagonal action, which will be also be by isometries. (This homomorphism
  is similar to the version of the modular homomorphism used in \cite{dcvl},
  which maps to $\R^n\rtimes GL_n(\R)$ but which does not specifically
  consider orthogonal matrices.)
  
  Suppose that the matrix
  $X\in GL(n,\R)$ is such that $X{\mathcal M}(G)X^{-1}$ lies in $O(n)$.
  If $e_1,\ldots ,e_n$ is the standard orthogonal basis for $\R^n$
  then set $\alpha_i$ to be the vector $X(e_i)$ in $\R^n$. We now let
  the base subgroup $A$ act via $\theta$
  on $\R^n$ as a group of translations, where we define
  $\theta(a_i)$ to be the map  $x\mapsto x+\alpha_i$. Recalling how we
  obtain a finite
  presentation for $G$, we next extend $\theta$ to all vertex subgroups
  by using the appropriate translations, where we map consistently
  across edge subgroups for the edges in some
  maximal tree, by working out from $v_0$. This works because any edge
  subgroup $E$ is a copy of $\Z^n$ with finite index in its neighbouring
  vertex subgroup $V$,
  so a homomorphism from $E$ to $\R^n$ has a unique extension from $V$ to
  $\R^n$.

  We now need to set $\theta(t)$
  for all stable letters $t$. Once we have chosen $\theta(t)$ to be some
  isometry of $\R^n$, $\theta$ will be well defined if $t$ acts by
  conjugation on $A$ in the correct way because $A$ has finite index in any
  vertex group.

  Therefore we send $t$ to the orthogonal matrix $X{\mathcal M}(t)X^{-1}$.
Let us write the matrix ${\mathcal M}(t)$ as 
  ${\mathcal T}$ with $i,j$th entry equal to $\tau_{ij}$.
  We then have that $\theta(t)\theta(a_j)\theta(t^{-1})$ is the translation
  by
  \[X{\mathcal T} X^{-1}(\alpha_j)=X{\mathcal T}(e_j)=
    X(\tau_{1j}e_1+\ldots +\tau_{nj}e_n)=\tau_{1j}\alpha_1+\ldots +\tau_{nj}
  \alpha_n.\]
  As for the translation given by $\theta(ta_jt^{-1})$, we can assume
  $ta_jt^{-1}$ is in $A$ by replacing $a_j$ with a power if necessary,
  whereupon we have
  \[ ta_jt^{-1}=a_1^{\tau_{1j}}\ldots a_n^{\tau_{nj}}\]
by definition of the modular homomorphism  
and so $\theta(ta_jt^{-1})$ will be the translation of $\R^n$ by
  $\tau_{1j}\alpha_1+\ldots +\tau_{nj}\alpha_n$ too.
      Thus $\theta$ is well defined and so $G$ acts on
    $T\times \R^n$ by isometries.

This action of $G$ on $T\times \R^n$ is (metrically) proper: for instance
see \cite{dcvl} Lemma 4.5. This implies that the action is free (and in
particular faithful) because $G$ is torsion free.
It is also cocompact, as we can take the compact
set $C\times D\subseteq T\times \R^n$ where $C$ is a closed ball centred
at $v_0$ and containing a fundamental domain for the action of $G$ on $T$,
whereas $D$ is a closed parallelopiped containing a fundamental domain
for the action of $A$ on $\R^n$. Given any point $(t,x)\in T\times \R^n$,
we can find $g\in G$ which moves $t$ into $C$, and then take some element
$a\in A$ which moves $g(x)$ into $D$. As $a\in Stab(v_0)$ we have $a(C)=C$
so that $ag(t,x)\in C\times D$.

For the converse, first suppose that $G$ is a CAT(0) group, acting
geometrically (or even properly and semi-simply) on the CAT(0) space
$X$. Then by \cite{lrmn}
Theorem 6.4 (4') (itself based on the flat torus theorem of \cite{bh}
Theorem II.7.1, but replacing elements that normalise a copy of $\Z^n$
with elements that commensurate it), there is an inner product on
$A\otimes \R\cong \R^n$ such that the monodromy of any isometry of $X$ that
commensurates $A$ preserves this inner product, so that we can conjugate
the monodromy of $G$ to lie in $O(n)$.

Finally if $G$ has a finite index subgroup $H$ that is a CAT(0) group
then by the above we can
conjugate the monodromy ${\mathcal M}(G)$ so that the finite index subgroup
${\mathcal M}(H)$ lies in $O(n)$. We can now further conjugate so that
all of ${\mathcal M}(G)$ is in $O(n)$ too (either by mimicking the proof
when ${\mathcal M}(H)$ is trivial, or by $O(n)$ being a maximal compact
subgroup of $GL(n,\R)$).
\end{proof}  

We can now bring in examples from \cite{lrmn}.
\begin{co} \label{hhglm} The Leary - Minasyan group
given by the finite presentation
\[L=\langle t,a,b\,|\,[a,b],ta^2b^{-1}t^{-1}=a^2b,tab^2t^{-1}=a^{-1}b^2\rangle \]
is a $GBS_2$ group which is not virtually a HHG but which is quasi-isometric
to a HHG.
\end{co}
\begin{proof}
We see that $L$ is a $GBS_2$ group with base
$\Z^2$-subgroup $\langle a,b\rangle$ by using
the graph of groups decomposition of one vertex and one edge associated
to the HNN extension, with the edge subgroup having index 5
in the vertex group at both ends. Thus the Bass - Serre tree is the regular
tree $T_{10}$.
The monodromy is determined by the conjugation
action of the single stable letter and we can work this out explicitly
by noting that the above relations imply that
\[ta^5t^{-1}=a^3b^4 \mbox{ and }tb^5t^{-1}=a^{-4}b^3\] both hold in $G$.
Thus the monodromy of $L$ is generated by the matrix
\[\sma{rr} \frac{3}{5}& -\frac{4}{5}\\  \frac{4}{5}&\frac{3}{5}\fma\]
which has infinite order. Hence $L$ has infinite monodromy and is not
virtually a HHG by Corollary \ref{nthhg}.

However this matrix is orthogonal so Theorem \ref{latt}
says that $L$ acts properly and cocompactly by isometries on $T_{10}\times \R^2$,
as does $M=F_5\times\Z^2$. Thus by Svarc - Milnor
$L$ and $M$ are both quasi-isometric to this space and hence to each other,
but $M$ is certainly a HHG.
\end{proof}

\section[(QT) is not invariant under quasi-isometry]{Property (QT) is not
  invariant under quasi-isometry}

Another property of groups which, like HHGs, considers how a
group can act on different hyperbolic spaces, is
Bestvina, Bromberg and Fujiwara's property (QT) from \cite{bebf}.
Here a quasitree will always be a graph equipped with its
path metric which is quasi-isometric to a simplicial tree but which need
not be locally finite. A finitely
generated group $G$ (equipped with the word metric with respect to some
finite generating set)
is said to have ({\bf QT}) if it acts by isometries
on a finite product $P$ of
quasitrees equipped with the $\ell_1$ product metric
such that the orbit map (using an arbitrary basepoint of $P$)
is a quasi-isometric embedding from $G$ to $P$.
This is a strong definition: for instance it implies that this 
action is metrically proper. Nevertheless it is shown in \cite{bebf}
that mapping class groups and all residually finite hyperbolic groups have
(QT). It is also a consequence of \cite{drnjan} that every Coxeter group has
(QT).

Moreover property (QT) has good closure properties, in fact these are better
than for HHGs. It is certainly the case that if $G$ has (QT) and $H$ has
finite index in $G$ then $H$ also has (QT) but the definition ensures that
it also holds more generally when $H$ is an undistorted finitely generated
subgroup of $G$. Moreover if $G_1$ has (QT) via an action on the space $P_1$
and $G_2$ on the space $P_2$ then it can be checked directly that
$G_1\times G_2$ has (QT) by letting it act on the direct product
$P_1\times P_2$ (with the $\ell_1$ product metric) using the action
on each factor and summing the word metrics on $G_1$ and $G_2$, which
is the word metric on $G_1\times G_2$ with the obvious generating set.

However property (QT) is also a commensurability invariant because
if $H$ has index $i$ in $G$ and $H$ acts isometrically on the product $P$ of
quasitrees then we can induce an isometric action of $G$ on the product $P^i$
of copies of $P$. This will also turn the orbit map under $G$
into a quasi-isometric embedding. In particular a group which
virtually has (QT) does itself have (QT). Therefore the questions
of whether possessing (QT) and virtually possessing (QT) are quasi-isometry
invariants are in fact the same question. Here we will answer this by first
considering how an isometry on a product of graphs
$\Gamma_1\times \ldots \times\Gamma_m$
(each equipped with the path metric and then using the $\ell_1$ product
metric) breaks up, or at least virtually breaks up, into individual
isometries on each $\Gamma_i$. That this can be done using the $\ell_\infty$
metric is a result of W.\,Malone in \cite{malw} and we will now mimic his
proof for the (easier) $\ell_1$ case.

\begin{thm}   \label{lonpr} 
Suppose that $X=\Gamma_1\times \ldots \times \Gamma_m$ is a finite product of
connected
graphs, where each $\Gamma_i$ has the induced path metric and $X$
has the $l_1$ or the $l_\infty$
product metric. Suppose that $G$ is any group
acting by isometries on $X$. Then $G$
has a finite index subgroup $H$ which preserves factors
and acts as an isometry on each factor.
\end{thm}
\begin{proof}
  On giving a graph the induced path metric, it becomes not just a
metric space which is geodesic
but a geodesic metric space which is locally uniquely geodesic.
Thus if $X$ is given the $l_\infty$ metric then this result is a direct
consequence of \cite{malw}, which is established by using the fact that for
certain directions in the product space, geodesics are unique if they are
unique in the factors. This is certainly true in an $l_1$ product space
if we travel ``horizontally/vertically'', so in this case we proceed
as follows:
\begin{lem} \label{unig}
  Let $X=X_1\times\ldots\times X_m$ have the $l_1$ product
  metric $d_X$
  for geodesic metric spaces $(X_1,d_1),\ldots ,(X_m,d_m)$.
Suppose we have  two points ${\bf x},{\bf y} \in X$ of the form
\[{\bf x}=(x_1,\ldots ,x_{i-1},x_i,x_{i+1},\ldots ,x_m)\mbox{ and }{\bf y}=
  (x_1,\ldots ,x_{i-1},y_i,x_{i+1},\ldots ,x_m)\]
for $1\leq i\leq m$. Then any geodesic between $\bf x$ and $\bf y$
only varies in the $i$th coordinate. Conversely if we have two
points ${\bf x},{\bf y} \in X$ which differ in at least two coordinates
then there is more than one geodesic from $\bf x$ to $\bf y$ in $X$.
\end{lem}
\begin{proof}
  Suppose we have a (unit speed) geodesic
  $\boldsymbol{\gamma}:[0,d]\rightarrow X$ from $\bf x$ to $\bf y$ where
  $d=d_X({\bf x},{\bf y})$. If there is $k\neq i$ and $t\in (0,d)$
  such that the $kth$ coordinate $u_k$ of
  $\boldsymbol{\gamma}(t)=(u_1,\ldots ,u_m)$ is not equal to $x_k$ then
  \begin{eqnarray*} d_i(x_i,y_i)=
    d_X({\bf x},{\bf y})&=&d({\bf x},\boldsymbol{\gamma}(t))
      +d(\boldsymbol{\gamma}(t),{\bf y})\\
                        &=&\Sigma_{j=1,j\neq i}^m \left(d_j(x_j,u_j)
                            +d_j(u_j,x_j)\right)
                            +d_i(x_i,u_i)+d_i(u_i,y_i)\\
                        &\geq& d_k(x_k,u_k)+d_k(u_k,x_k)+
                               d_i(x_i,u_i)+d_i(u_i,y_i)\\
                          &>& d_i(x_i,y_i)\mbox { as } x_k\neq u_k
\end{eqnarray*}
which is a contradiction.

Now suppose we have two points ${\bf x}=(x_1,\ldots ,x_m)$ and
${\bf y}=(y_1,\ldots ,y_m)$ in $X$ which differ in (without loss of generality)
at least the first two coordinates. Take constant speed
geodesics $\gamma_i$ in each $X_i$ running from $x_i$ to $y_i$ and set
\[\boldsymbol{\delta}_1(t)=(x_1,\gamma_2(t),\ldots ,\gamma_m(t)),
  \boldsymbol{\delta}_2(t)=(\gamma_1(t),y_2,\ldots ,y_m))\]
and
\[\boldsymbol{\delta}_3(t)=(\gamma_1(t),x_2,\gamma_3(t),\ldots ,\gamma_m(t)),
  \boldsymbol{\delta}_4(t)=(y_1,\gamma_2(t),y_3,\ldots ,y_m))\]
where in $\boldsymbol{\delta}_1$ the geodesics
$\gamma_2,\ldots ,\gamma_m$ are reparametrised to have domain
$[0,d_X(({\bf x},{\bf y})-d_1(x_1,y_1)]$ but $\gamma_1$ and hence
$\boldsymbol{\delta}_2$ remain unit speed. We also do the same
for $\boldsymbol{\delta}_3$ and $\boldsymbol{\delta}_4$.
Then it is easily checked that following $\boldsymbol{\delta}_1$ then
$\boldsymbol{\delta}_2$
and also following $\boldsymbol{\delta}_3$ then $\boldsymbol{\delta}_4$
are both unit speed geodesics
from $\bf x$ to $\bf y$. Moreover they are distinct as 
the first geodesic passes through a point that projects to $(x_1,y_2)$
in the first two coordinates, whereas the second geodesic never does 
since $x_1\neq y_1$ and $x_2\neq y_2$.
\end{proof}

Now we return to the setting in Theorem \ref{lonpr}. Each $\Gamma_i$
is a geodesic metric space (which we
assume without loss of generality is not a single point). Thus given
any isometry $g$ of $X$, take any
point ${\bf x}=(x_1,\ldots ,x_i,\ldots ,x_m)$ in $X$ and for a given
$1\leq i\leq m$,
let $z_i$ be another
point in $\Gamma_i$ near $x_i$ such that there is only one geodesic
$\gamma_i$ in $\Gamma_i$ from $x_i$ to $z_i$. Then by Lemma \ref{unig}
we have that 
$\boldsymbol{\gamma}(t)$, which is equal to
$(x_1,\ldots ,\gamma_i(t),x_{i+1},\ldots ,x_m)$,
is the unique geodesic in $X$ between $\bf x$ and
${\bf z}=
(x_1,\ldots ,x_{i-1},z_i,x_{i+1},\ldots ,x_m)$ and so under $g$
it must map to
a unique geodesic between $g({\bf x})$ and $g({\bf z})$, by considering
$g^{-1}$. Thus by Lemma \ref{unig} again
these two points (which are not the same because $g$ is
a bijection) differ in only one coordinate, say the $j$th.

Now take $y_i$ to be an arbitrary point in $\Gamma_i$. Given any geodesic
in $\Gamma_i$ from $x_i$ to $y_i$, we can split it into a finite number
of subgeodesics, with some overlap that is more than a point, where each
subgeodesic is the unique geodesic between its own endpoints. Thus the image
of each of these subgeodesics under $g$ is a subset of $X$ where only
one coordinate varies, and as these subgeodesics overlap this will always
be the $j$th coordinate. In other words given $1\leq i\leq m$ we have
$1\leq j\leq m$ such that for fixed
$x_1\in \Gamma_1,\ldots ,x_{i-1}\in \Gamma_{i-1},x_{i+1}\in\Gamma_{i+1},\ldots ,
x_m\in\Gamma_m$ there is
$y_1\in \Gamma_1,\ldots ,y_{j-1}\in \Gamma_{j-1},y_{j+1}\in\Gamma_{j+1},\ldots ,
y_m\in \Gamma_m$ and a function $f:\Gamma_i\rightarrow \Gamma_j$
with

\[g(x_1,\ldots ,x_{i-1},x,x_{i+1},\ldots ,x_m)=
  (y_1,\ldots ,y_{j-1},f(x),y_{j+1},\ldots ,y_m)\]
for all $x\in \Gamma_i$. Moreover $f$ is a
bijection and an isometry because $g$ is. By varying over all $m$
coordinates we see that $g$ permutes $\Gamma_1,\ldots ,\Gamma_m$ and for each
$1\leq k\leq m$ it acts as an isometry from $\Gamma_k$ to $\Gamma_{g(k)}$
(but clearly only isometric factors can be permuted). We can therefore
set
$H$ to be the kernel of this finite permutation action, so that $H$ has
finite index in $G$, preserves factors and acts as an isometry on each
factor.
\end{proof}                              

We can now use the same example as for HHGs to show that property (QT)
is not a quasi-isometric invariant.
\begin{thm} \label{ashl}
A $GBS_n$ $G$ group with infinite monodromy does not
possess any metrically proper action by isometries
  on any finite product of graphs (using the path metric)
  which are quasitrees with the $l_1$ (or $l_\infty$) product metric.
  In particular $G$ does not have property (QT).
  \end{thm}
  \begin{proof} If $G$ did have such an action then
by Theorem \ref{lonpr} $G$ would have a finite index
subgroup $H$ (also acting metrically
properly) which preserves factors
and acts by isometries on each factor. But $H$ is also a $GBS_n$ group
with infinite monodromy. Thus by Theorem \ref{act} applied to $H$,
there is a non trivial
element $z$ of $H$ which does not act loxodromically in any action of
$H$ on a hyperbolic space. But as isometries of quasitrees
can only be loxodromic or elliptic by \cite{man}
the action of the infinite order element $z$ is
elliptic on each quasitree.
Hence the action of $\langle z\rangle$ on $X$
is also bounded, so we have an infinite subgroup of $G$ acting on $X$
with bounded orbits, meaning that the action of $G$ is not
metrically proper. Therefore the orbit map from $G$ to $X$ cannot be a
quasi-isometric embedding because because there are only finitely
many elements in $G$ with word length at most a given value.
\end{proof}
\begin{co} \label{qt}
  The property of having (QT) is not preserved under quasi-isometries.
\end{co}
\begin{proof} We take
  the Leary - Minasyan group $L$ as given in Corollary \ref{hhglm}
  and again note that it has infinite monodromy, so it
  does not  have (QT) by Theorem \ref{ashl}. However it is quasi-isometric
    to $M=\Z^2\times F_5$ which acts geometrically on the product
    $P=\R\times\R\times T_{10}$ of
    simplicial trees and hence the orbit map from $M$ to $P$ is a
    quasi-isometry by Svarc - Milnor.
  \end{proof}
Note that as $P$ is a product not just of arbitrary quasitrees but of
bounded valence simplicial trees, changing the property (QT) by
replacing quasitrees with trees or inserting the bounded valence
condition, in any combination, will still not result in a quasi-isometry
invariant.

Note also that
  if we equip our product spaces with the $\ell_2$ product metric
  then the Leary - Minasyan group does indeed act geometrically on
  a product of quasi-trees with the orbit map a quasi-isometric embedding.
  Here the equivalent of Theorem \ref{lonpr} does not hold as we can
  have a de Rham factor (see \cite{foly} for this case). However the
  definition of (QT) uses the $\ell_1$ product metric.

  Finally we also note that the Leary - Minasyan group acts geometrically
  on the CAT(0) cube complex $P$ by {\it isometries} using the CAT(0)
metric  but it has no geometric
  action on any CAT(0) cube complex by {\it cubical automorphisms} and nor
  does any finite index subgroup. This is noted in \cite{lrmn} as a 
consequence
  of \cite{hupr}.

\Address


\begin{thebibliography}{99}

\bibitem{abos} C.\,Abbott, S.\,H.\,Balasubramanya and D.\,Osin,
{\it Hyperbolic structures on groups},
Algebraic and Geometric Topology {\bf 19} (2019) 1747--1835.


\bibitem{abs} C.\,Abbott, T.\,Ng, D.\,Spriano, R.\,Gupta and H.\,Petyt,
{\it Hierarchically hyperbolic groups and uniform exponential growth},
Math. Z. {\bf 306} (2024), Paper No. 18.

\bibitem{beekjgt} B.\,Beeker,
  {\it Abelian JSJ decomposition of graphs of free abelian groups},
  J. Group Theory {\bf 17} (2014) 337--359.
  
\bibitem{bhs1}
  J.\,Behrstock, M.\,F.\,Hagen and A.\,Sisto,
  {\it Hierarchically hyperbolic spaces, I: Curve complexes for cubical
    groups}, Geom. Topol. {\bf 21} (2017) 1731--1804.

\bibitem{bhs2}
  J.\,Behrstock, M.\,F.\,Hagen and A.\,Sisto,
{\it  Hierarchically hyperbolic spaces II: Combination theorems and the 
distance formula}, Pacific J. Math. {\bf 299} (2019) 257--338. 


\bibitem{bebf}  M.\,Bestvina, K.\,Bromberg and K.\,Fujiwara,
{\it Proper actions on finite products of quasi-trees},
 Ann. H. Lebesgue {\bf 4} (2021) 685--709. 

\bibitem{arm} O.\,Bogopolski, A.\,Martino and E.\,Ventura, {\it Orbit
    decidability and the conjugacy problem for some extensions of groups},
Trans. Amer. Math. Soc. {\bf 362} (2010) 2003--2036. 
 
\bibitem{bow} B.\,Bowditch,
{\it Tight geodesics in the curve complex},
Invent. Math. {\bf 171} (2008) 281--300. 

\bibitem{bh} M.\,R.\,Bridson and A.\,Haefliger,
Metric spaces of non-positive curvature,
Springer-Verlag, Berlin, 1999.

\bibitem{ccmt} P-E.\,Caprace, Y.\,de Cornulier, N.\,Monod and R.\,Tessera,
  {\it Amenable hyperbolic groups},
  J.\,Eur.\,Math.\,Soc. {\bf 17} (2015) 2903--2947.

\bibitem{mntil} M.\,Casals-Ruiz, I.\,Kazachkov and A.\,Zakharov,
{\it Commensurability of Baumslag-Solitar groups},
Indiana Univ. Math. J. {\bf 70} (2021) 2527--2555.

\bibitem{cl} C.\,H.\,Cashen and G.\,Levitt,
{\it Mapping tori of free group automorphisms, and the Bieri-Neumann-Strebel
invariant of graphs of groups}, J. Group Theory {\bf 19} (2016) 191--216.  

\bibitem{dects}  Y.\,de Cornulier and R.\,Tessera,
  {\it Metabelian groups with quadratic Dehn function and Baumslag-Solitar
    groups},
Confluentes Math. {\bf 2} (2010) 431--443. 

\bibitem{dcvl}
  Y.\,de Cornulier and A.\,Valette,
  {\it On equivariant embeddings of generalized Baumslag-Solitar groups},
  Geom. Dedicata {\bf 175} (2015) 385--401. 

\bibitem{ren} R.\,B.\,Coulon
  {\it Partial periodic quotients of groups acting on a hyperbolic space},
Ann. Inst. Fourier (Grenoble) {\bf 66} (2016) 1773--1857.


\bibitem{drnjan} A.\,Dranishnikov and T.\,Januszkiewicz,
{\it Every Coxeter group acts amenably on a compact space},
Topology Proc. {\bf 24} (1999) 135--141.

\bibitem{dhs} M.\,G.\,Durham, M.\,F.\,Hagen and A.\,Sisto,
  {\it Boundaries and automorphisms of hierarchically hyperbolic spaces},
  Geom. Topol. {\bf 21} (2017) 3659--3758. 

\bibitem{dhsc} M.\,G.\,Durham, M.\,F.\,Hagen and A.\,Sisto,
  {\it Correction to the article Boundaries and automorphisms of hierarchically
    hyperbolic spaces}, Geom. Topol. {\bf 24} (2020) 1051--1073. 

 
\bibitem{foly} T.\,Foertsch and A.\,Lytchak,
{\it The de Rham decomposition theorem for metric spaces},
Geom. Funct. Anal. {\bf 18} (2008) 120--143.

\bibitem{fordef} M.\,Forester,
  {\it Deformation and rigidity of simplicial group actions on trees},
 Geom. Topol. {\bf 6} (2002) 219--267.

\bibitem{hag} M.\,F.\,Hagen,
  {\it Cocompactly cubulated crystallographic groups},
  J. Lond. Math. Soc. {\bf 90} (2014) 140–166.

 
\bibitem{hrsspr}  M.\,Hagen, J.\,Russell, A.\,Sisto and D.\,Spriano,
{\it Equivariant hierarchically hyperbolic structures for 3-manifold groups 
via quasimorphisms},
  \texttt{https://arxiv.org/abs/2206.12244v1}

\bibitem{delah}
P.\,de la Harpe,
Topics in geometric group theory. 
Chicago Lectures in Mathematics. University of Chicago Press, Chicago, IL,
2000. 

\bibitem{hupr}
  J.\,Huang and T.\,Prytuła,
  {\it Commensurators of abelian subgroups in CAT(0) groups},
  Math. Z. {\bf 296} (2020) 79--98.
  
\bibitem{hghv}  S.\,Hughes and M.\,Valiunas,
  {\it Commensurating HNN-extensions: hierarchical hyperbolicity and
    biautomaticity},
  \texttt{https://arxiv.org/abs/2203.11996v1}

\bibitem{kou} M.\, Koubi,
{\it Croissance uniforme dans les groupes hyperboliques},
Ann. Inst. Fourier (Grenoble) {\bf 48} (1998) 1441--1453. 
  
\bibitem{krpp} P.\,H.\,Kropholler,
{\it Baumslag-Solitar groups and some other groups of
cohomological dimension two},
Comment. Math. Helvetici {\bf 65} (1990) 547--558.
  
\bibitem{lrmn} I.\,J.\,Leary and A.\,Minasyan,
  {\it Commensurating HNN extensions: nonpositive curvature and
    biautomaticity},
  Geom. Topol. {\bf 25} (2021) 1819--1860.

\bibitem{lev}  G.\,Levitt, {\it On the automorphism group of generalized
    Baumslag-Solitar groups}, Geom. Topol. {\bf 11} (2007) 473--515.


\bibitem{malw} W.\,Malone,
  {\it Isometries of products of path-connected locally uniquely geodesic
    metric spaces with the sup metric are reducible},
  \texttt{https://arxiv.org/abs/0912.3471v1}


\bibitem{man} J.\,F.\,Manning,
{\it Quasi-actions on trees and property (QFA)},
J. London Math. Soc. {\bf 73} (2006) 84--108.

\bibitem{osmn} A.\,Minaysan and D.\,Osin,
{\it Acylindrical hyperbolicity of groups acting on trees},
Math. Ann. {\bf 362} (2015) 1055--1105.

\bibitem{osmnc} A.\,Minasyan and D.\,Osin,
  {\it Correction to: Acylindrical hyperbolicity of groups acting on trees},
  Math. Ann. 373 (2019) 895--900.
  
\bibitem{osac}  D.\,Osin
{\it Acylindrically hyperbolic groups},
Trans. Amer. Math. Soc. {\bf 368} (2016) 851--888.

\bibitem{petspr} H.\,Petyt and D.\,Spriano,
  {\it Unbounded domains in hierarchically hyperbolic groups},
Groups Geom. Dyn. {\bf 17} (2023) 479--500.
  
\end{thebibliography}
\end{document}